\documentclass[12pt]{amsart}
\usepackage[margin=1.25in]{geometry}
\usepackage[utf8]{inputenc}
\usepackage{amsmath, amssymb, amsthm}
\usepackage{enumerate}
\usepackage{algorithm}
\usepackage[noend]{algpseudocode}
\usepackage[foot]{amsaddr}

\title{Tree decompositions and many-sided separations}
\author{Tara Abrishami $^{\ast \dagger}$}
\address{$^{\ast}$Princeton University, Princeton, NJ, USA}
\address{$^{\dagger}$ Supported by NSF Grant DMS-1763817 and
     NSF-EPSRC Grant DMS-2120644.}

\renewcommand{\S}{\mathcal{S}}
\newtheorem{lemma}{Lemma}
\newtheorem{theorem}{Theorem}

\newcounter{tbox}

\newcommand{\sta}[1]{\vspace*{0.3cm}\refstepcounter{tbox}\noindent{ \parbox{\textwidth}{(\thetbox) \emph{#1}}}\vspace*{0.3cm}}

\begin{document}

\maketitle
\begin{abstract}
    A {\em separation} of a graph $G$ is a partition $(A_1, A_2, C)$ of $V(G)$ such that $A_1$ is anticomplete to $A_2$. A classic result from Robertson and Seymour's Graph Minors Project states that there is a correspondence between tree decompositions and laminar collections of separations. A {\em many-sided separation} of a graph $G$ is a partition $(A_1, \hdots, A_k, C)$ of $V(G)$ such that $A_i$ is anticomplete to $A_j$ for all $1 \leq i < j \leq k$. In this note, we show a correspondence between tree decompositions with a certain parity property, called {\em deciduous} tree decompositions, and laminar collections of many-sided separations. 
\end{abstract}
\section{Introduction}

A {\em tree decomposition} $(T, \chi)$ of a graph $G$ consists of a tree $T$ and a map $\chi: V(T) \to 2^{V(G)}$, satisfying the following properties: 
\begin{enumerate}[\hspace{1.5em} (i)]
    \item For every $v \in V(G)$, there exists $t \in V(T)$ such that $v \in \chi(t)$, 
    \item For every $v_1v_2 \in E(G)$, there exists $t \in V(T)$ such that $v_1, v_2 \in \chi(t)$, 
    \item For every $v \in V(G)$, the support of $v$ in $(T, \chi)$ is connected. 
\end{enumerate}

The {\em width} of a tree decomposition $(T, \chi)$ is $\min_{t \in V(T)} |\chi(t)| - 1$. The {\em treewidth} of a graph $G$ is the minimum width of a tree decomposition of $G$. The sets $\chi(t)$ for $t \in V(T)$ are called the {\em bags} of the tree decomposition $(T, \chi)$. For a set $X \subseteq V(T)$, we let $\chi(X) = \bigcup_{t \in X} \chi(t)$. 

Let $G$ be a graph and let $X, Y \subseteq V(G)$. We say {\em $X$ is anticomplete to $Y$} if no edge of $G$ has one end in $X$ and one end in $Y$. A {\em separation} of $G$ is a triple $(A_1, A_2, C)$ such that: 1) $(A_1, A_2, C)$ is a partition of $V(G)$, i.e. $A_1, A_2, C \subseteq V(G)$, $A_1, A_2, C$ are pairwise disjoint, and $A_1 \cup A_2 \cup C = V(G)$; and 2) $A_1$ is anticomplete to $A_2$. Two separations $(A_1, A_2, C)$ and $(B_1, B_2, D)$ are {\em non-crossing} if there exist $1 \leq i, j \leq 2$ such that $A_{3-i} \cup C \subseteq B_{j} \cup D$ and $B_{3-j} \cup D \subseteq A_i \cup C$. 

Let $G$ be a graph and let $(T, \chi)$ be a tree decomposition of $G$. Let $e = t_1t_2$ be an edge of $E(T)$ and let $T_1, T_2$ be the connected components of $T \setminus \{e\}$ containing $t_1$ and $t_2$, respectively. Then, up to symmetry between $t_1$ and $t_2$, the {\em separation corresponding to $e$}, denoted $S_e = (A_1^e, A_2^e, C^e)$, is defined as follows: $C^e = \chi(t_1) \cap \chi(t_2)$, $A_1^e = \chi(T_1) \setminus C^e$, and $A_2^e = \chi(T_2) \setminus C^e$. The collection of separations corresponding to $(T, \chi)$, denoted $\tau((T, \chi))$ is defined as follows: $\tau((T, \chi)) =\{S_e \mid e \in E(T)\}$. 

A collection $\S$ of separations of $G$ is called {\em laminar} if the separations of $\S$ are pairwise non-crossing. The following result of Robertson and Seymour states that there is a correspondence between tree decompositions of $G$ and laminar collections of separations of $G$. 

\begin{theorem}[\cite{GMX}]
\label{thm:RS}
Let $G$ be a graph. Then, for every tree decomposition $(T, \chi)$ of $G$, $\tau((T, \chi))$ is laminar. Conversely, for every laminar collection of separations $S$ of $G$, there exists a tree decomposition $(T, \chi)$ of $G$ such that $\tau((T, \chi)) = S$.  
\end{theorem}

In this note, we show that a similar correspondence exists for generalizations of separations called many-sided separations. 
A {\em many-sided separation of $G$} is a tuple $(A_1, A_2, \hdots, A_k, C)$, with $k \geq 2$, such that: 

\begin{enumerate}[\hspace{1.5em} (i)]
    \item $A_1, A_2, \hdots, A_k, C \subseteq V(G)$,
    \item $A_1, A_2, \hdots, A_k, C$ are pairwise disjoint, 
    \item $A_1 \cup \hdots \cup A_k \cup C = V(G)$, and
    \item for all $1 \leq i < j \leq k$, $A_i$ is anticomplete to $A_j$. 
\end{enumerate}

Many-sided separations offer a way to store more precise structural information than standard separations. To see this, consider the following example. Let $G$ be a connected graph and let $C \subseteq V(G)$ be such that $G \setminus C$ has $k$ connected components $A_1, \hdots, A_k$. The many-sided separation given by $(A_1, \hdots, A_k, C)$ includes each of the connected components of $G \setminus C$, thus storing full information about the structure of $G \setminus C$. This can also be represented by a standard separation as follows: choose $\emptyset \subsetneq I \subset \{1, \hdots, k\}$, let $A_1' = \bigcup_{i \in I} A_i$, and let $A_2' = \bigcup_{j \in \{1, \hdots, k\} \setminus I} A_j$. Now, $(A_1', A_2', C)$ is a standard separation containing information about $G \setminus C$. Note that while the many-sided separation $(A_1, \hdots, A_k, C)$ is unique (up to symmetry between the order of the components), there are $2^{k-1} - 1$ possibilities for the standard separation $(A_1', A_2', C)$ (up to symmetry between $A_1'$ and $A_2'$). Further, each $A_i$ is connected in the many-sided separation $(A_1, \hdots, A_k, C)$, but no choice of $I$ allows for both $A_1'$ and $A_2'$ to be connected in the standard separation $(A_1', A_2', C)$. 

Separations are particularly powerful due to their relationship with deep structural concepts such as tree decompositions. In this paper, we show that there is a natural connection between many-sided separations and certain tree decompositions, extending the idea of Theorem \ref{thm:RS} to many-sided separations.

\section{Many-sided separations} 

We begin this section with some definitions related to many-sided separations. Two many-sided separations $(A_1, A_2, \hdots, A_k, C)$ and $(B_1, B_2, \hdots, B_\ell, D)$ are {\em non-crossing} if there exist $1 \leq i \leq k$ and $1 \leq j \leq \ell$, such that 
\begin{itemize} 

\item $\left(\bigcup_{1 \leq p \leq \ell} B_p \setminus B_j\right) \cup D \subseteq A_i \cup C$, and 
\item $\left(\bigcup_{1 \leq p \leq k} A_p \setminus A_i\right) \cup C \subseteq B_j \cup D$. 
\end{itemize}
Notice that when $k = \ell = 2$, the definition of non-crossing for many-sided separations is the same as the definition of non-crossing for separations. A collection $\S$ of many-sided separations is called {\em laminar} if the many-sided separations of $\S$ are pairwise non-crossing. 

We say that a tree $T$ is {\em deciduous} if for every two leaves $\ell_1, \ell_2$ of $T$, the path from $\ell_1$ to $\ell_2$ in $T$ has even length. A tree decomposition $(T, \chi)$ of a graph $G$ is {\em deciduous} if $T$ is deciduous. Let $G$ be a graph and let $(T, \chi)$ be a deciduous tree decomposition of $G$. Let $(X, Y)$ be a bipartition of $V(T)$ and assume up to symmetry that every leaf of $T$ is in $Y$. For $x \in X$, the {\em many-sided separation corresponding to $x$}, denoted $S_x = (A_1^x, \hdots, A_{\text{deg}(x)}^x, C)$, is defined as follows. Let $\deg(x) = k$ and let $y_1, \hdots, y_k$ be the neighbors of $x$ in $T$. Let $T_1, \hdots, T_k$ be the components of $T \setminus \{t\}$ containing $y_1, \hdots, y_k$, respectively. Then, $C^x = \chi(t)$, and for $1 \leq i \leq k$, $A_i^x = \chi(T_i) \setminus \chi(t)$. Now, the collection of many-sided separations corresponding to $(T, \chi)$, denoted $\tau^*((T, \chi))$, is defined as follows: $\tau^*((T, \chi)) = \{S_x \mid x \in X\}$. 

The goal of the next two theorem is to prove a result similar to Theorem \ref{thm:RS} for many-sided separations and deciduous tree decompositions. 


\begin{theorem}
\label{thm:S-to-TD}
Let $G$ be a graph and let $\S$ be a laminar collection of many-sided separations of $G$. Then, there exists a deciduous tree decomposition $(T, \chi)$ of $G$ such that $\tau^*((T, \chi)) = \S$. 
\end{theorem}
\begin{proof}
We proceed by induction on $|\S|$. Suppose $|\S| = 1$ and let $\S = \{S\}$, where $S = (A_1, \hdots, A_k, C)$. Let $T = K_{1, k}$ with vertex $u$ of degree $k$ and vertices $v_1, \hdots, v_k$ adjacent to $u$, and let $\chi: V(T) \to 2^{V(G)}$ be defined as follows: $\chi(u) = C$ and $\chi(v_i) = A_i \cup C$ for $1 \leq i \leq k$. It is easy to verify that $(T, \chi)$ is a deciduous tree decomposition of $G$ and that $\tau^*((T, \chi)) = \S$. 

Now, suppose the statement holds for all collections $\S'$ with $|\S'| = n-1$, and consider a collection $\S$ with $|\S| = n$. Let $\S = \{S_1, \hdots, S_n\}$. Since $\S$ is laminar, there exists $S = (A_1, \hdots, A_k, C) \in \S$ and $1 \leq i \leq k$ such that for all $S' = (A_1', \hdots, A_{k'}', C
) \in \S \setminus \{S\}$, there exists $1 \leq j \leq k'$ with $\left(\bigcup_{1 \leq p \leq k'} A_p' \setminus A_j \right) \cup C' \subseteq A_i \cup C$. Let $S^* \in \S$ be such that the set $A_j' \cup C'$ with this property is maximal.

Let $\S' = \S \setminus \{S\}$. Let $(T', \chi')$ be a deciduous tree decomposition of $G$ such that $\tau^*((T', \chi')) = \S'$. Let $(X', Y')$ be a bipartition of $V(T')$ such that all the leaves of $T$ are in $Y$, let $x^* \in V(T')$ be such that $S_{x^*} = S^*$, and let $j$ be the neighbor of $x^*$ in $T$ such that if $T_j$ is the subtree of $T \setminus \{x^*\}$ containing $j$, then $\chi(T_j) \setminus \chi(x^*) = A_j'$. We create tree $T$ from $T'$ as follows: add vertex $x$ adjacent to $j$, and add vertices $y_1, \hdots, y_k$ adjacent to $x$ for $1 \leq p \neq i \leq k$. We create $\chi$ from $\chi'$ as follows: $\chi(v) = \chi'(v)$ for all $v \in V(T')$, $\chi(x) = C$, and $\chi(y_p) = A_p \cup C$ for all $1 \leq p \neq i \leq k$. Now, it is easy to verify that $(T, \chi)$ is a deciduous tree decomposition of $G$ with $\tau^*((T, \chi)) = \S$. 
\end{proof}

\begin{theorem}
Let $G$ be a graph and let $(T, \chi)$ be a deciduous tree decomposition of $G$. Then, $\tau^*((T, \chi))$ is laminar. 
\end{theorem}
\begin{proof}
Let $S_1, S_2 \in \tau^*((T, \chi))$, where $S_1 = S_{x_1}$ and $S_2 = S_{x_2}$ for $x_1, x_2 \in X$. Let $y_1, \hdots, y_k$ be the neighbors of $x_1$ in $T$, and let $T_1, \hdots, T_k$ be the components of $T \setminus \{x_1\}$ containing $y_1, \hdots, y_k$, respectively. Let $i$ be such that $x_2 \in T_i$. Similarly, let $y_1', \hdots, y_{k'}'$ be the neighbors of $x_2$ in $T$, and let $T_1', \hdots, T_{k'}'$ be the components of $T \setminus \{x_2\}$ containing $y_1', \hdots, y_{k'}'$, respectively. Let $j$ be such that $x_1 \in T_{j}'$.  

Write $S_1 = (A_1, \hdots, A_k, C)$ and $S_2 = (A_1', \hdots, A_{k'}', C')$. We prove: 

\sta{\label{main-arg} For all $1 \leq p \leq k'$ with $p \neq j$, $A_p' \cup C' \subseteq A_i \cup C$.}

Since $x_2 \in T_i$ and $x_1 \in T_j'$, it follows that $T_p' \subseteq T_i$, and therefore that $\chi(T_p') \subseteq \chi(T_i)$. So $A_p' \cup C' \subseteq \chi(T_p') = A_i \cup C$. This proves \eqref{main-arg}. \\

By symmetry, it also holds that for all $1 \leq p \leq k$ with $p \neq i$, $A_p \cup C \subseteq A_j' \cup C'$. Therefore, $S_1$ and $S_2$ are non-crossing. This completes the proof of the lemma. 
\end{proof}

\section{Separations again}
In this section, we consider the following question. Let $G$ be a graph, let $\mathcal{S}^*$ be a laminar collection of many-sided separations of $G$, and let $(T, \chi)$ be a deciduous tree decomposition of $G$ such that $\tau^*((T, \chi)) = \mathcal{S}^*$ (as in Theorem \ref{thm:S-to-TD}). Since $(T, \chi)$ is a tree decomposition of $G$, it holds by Theorem \ref{thm:RS} that there is a laminar collection of (standard) separations $\mathcal{S}$ such that $\tau((T, \chi)) = \mathcal{S}$. What is the relationship between $\mathcal{S}^*$ and $\mathcal{S}$?

Let $S = (A_1, A_2, \hdots, A_k, C)$ be a many-sided separation of a graph $G$. The {\em separation projection} $\rho(S)$ of $S$ is a collection of $k$ separations defined as follows: $$\rho(S) = \{(A_i, \bigcup_{1 \leq j \neq i \leq k} A_j, C) \mid 1 \leq i \leq k\}.$$ Let $\S$ be a laminar collection of many-sided separations. The {\em separation projection $\rho(\S)$ of the collection $\S$} is defined as follows: 
$$\rho(\S) = \bigcup_{S \in \S} \rho(S).$$
\begin{lemma}
Let $G$ be a graph, let $\S$ be a laminar collection of many-sided separations of $G$, and let $\rho(\S)$ be the separation projection of $\S$. Then, $\rho(\S)$ is laminar.  
\end{lemma}
\begin{proof}
First, we show that $\rho(S)$ is laminar for $S \in \S$. Let $S = (A_1, \hdots, A_k, C)$ and let $S_1, S_2 \in \rho(S)$, with $S_1 = (A_1, \bigcup_{2 \leq i \leq k} A_i, C)$ and $S_2 = (A_2, \bigcup_{1 \leq i \neq 2 \leq k} A_i, C)$. Then, $A_2 \cup C \subseteq \left(\bigcup_{2 \leq i \leq k} A_i\right) \cup C$ and $A_1 \cup C \subseteq \left(\bigcup_{1 \leq i \neq 2 \leq k} A_i\right) \cup C$, so it follows that $S_1$ and $S_2$ are non-crossing. Therefore, $\rho(S)$ is laminar. 

Now, let $S, S' \in \S$ with $S = (A_1, \hdots, A_k, C)$ and $S' = (B_1, \hdots, B_\ell, D)$. Since $\S$ is laminar, it holds that $S, S'$ are non-crossing; let $1 \leq i \leq k$ and $1 \leq j \leq \ell$ be such that $\left(\bigcup_{1 \leq p \leq \ell} B_p \setminus B_j\right) \cup D \subseteq A_i \cup C$, and  $\left(\bigcup_{1 \leq p \leq k} A_p \setminus A_i\right) \cup C \subseteq B_j \cup D$. Let $S_1 = (A_1, \bigcup_{2 \leq p \leq k} A_p, C)$ and let $S_2 = (B_1, \bigcup_{2 \leq p \leq \ell} B_p, D)$. There are four cases. 
\begin{itemize} 
\item If $i \neq 1$, $j \neq 1$, then it follows that $A_1 \cup C \subseteq \bigcup_{2 \leq p \leq \ell} B_p \cup D$ and $B_1 \cup D \subseteq \bigcup_{2 \leq p \leq k} A_p \cup C$, so $S_1$ and $S_2$ are non-crossing. 

\item If $i \neq 1$, $j = 1$, then it follows that $A_1 \cup C \subseteq B_1 \cup D$ and $\bigcup_{2 \leq p \leq \ell} B_p \cup D \subseteq \bigcup_{2 \leq p \leq k} A_p \cup C$, so $S_1$ and $S_2$ are non-crossing. 

\item If $i =1$, $j \neq 1$, then it follows that $\bigcup_{2 \leq p \leq k} A_p \cup D \subseteq \bigcup_{2 \leq p \leq \ell} B_p \cup C$ and $B_1 \cup D \subseteq A_1 \cup C$, so $S_1$ and $S_2$ are non-crossing. 

\item Finally, if $i=1, j=1$, it follows that $\bigcup_{2 \leq p \leq k} A_p \cup C \subseteq B_1 \cup D$ and $\bigcup_{2 \leq p \leq \ell} B_p \cup D \subseteq A_1 \cup C$, so $S_1$ and $S_2$ are non-crossing. 
\end{itemize} 
We have shown that for all $S_1, S_2 \in \rho(\S)$, $S_1$ and $S_2$ are non-crossing, so $\rho(\S)$ is laminar. 
\end{proof}

Let $G$ be a connected graph. A set $C \subseteq V(G)$ is a {\em cutset of $G$} if $G \setminus C$ has at least two connected components. A cutset $C$ is {\em minimal} if for all $C' \subset C$, it holds that $G \setminus C'$ is connected. The following easy lemma states a property of minimal cutsets.

\begin{lemma}
Let $G$ be a connected graph and let $C$ be a minimal cutset of $G$. Let $A_1, \hdots, A_k$ be the connected components of $G \setminus C$, and let $c \in C$. Then, $c$ has a neighbor in $A_i$ for all $1 \leq i \leq k$. 
\end{lemma}
\begin{proof}
Suppose that $c$ has no neighbor in $A_j$ for some $1 \leq j \leq k$. Let $C' = C \setminus \{c\}$ and consider $G \setminus C'$.  Since $A_j$ is a connected component of $G \setminus C$, it follows that $A_j$ is anticomplete to $G \setminus C$. Further, $c$ has no neighbor in $A_j$, it follows that $A_j$ is anticomplete to $G \setminus C'$. Therefore, $A_j$ is a connected component of $G \setminus C'$. Since $A_j \neq V(G) \setminus C'$, it follows that $G \setminus C'$ is not connected, contradicting that $C$ is a minimal cutset of $G$. This completes the proof.
\end{proof}

We can now prove the following theorem. 
\begin{theorem}
\label{thm:deciduous-minimal}
Let $G$ be a connected graph and let $\S$ be a laminar collection of many-sided separations of $G$. Suppose that for every $S \in \S$ with $S = (A_1, \hdots, A_k, C)$, it holds that $C$ is a minimal cutset of $G$ and $A_1, \hdots, A_k$ are the connected components of $G \setminus C$. Let $(T, \chi)$ be the deciduous tree decomposition of $G$ such that $\tau^*((T, \chi)) = \S$. Then, $\tau((T, \chi)) = \rho(\S)$. 
\end{theorem}
\begin{proof}
Let $(X, Y)$ be a bipartition of $T$ and let $Y$ be the partite set containing the leaves of $T$. Then, every edge of $T$ is of the form $xy$ for $x \in X$, $y \in Y$. Let $x \in X$ and let $y_1, \hdots, y_k$ be the neighbors of $x$ in $T$, and let $S_x = (A_1, \hdots, A_k, C)$. 

\sta{\label{step-1} $\{S_{xy_i} \mid 1 \leq i \leq k\} = \rho(S_x)$.} 

Let $e = xy_i$ and let $T_1$, $T_2$ be the components of $T \setminus \{e\}$ containing $x$ and $y_i$, respectively. Then, $C^e = \chi(x) \cap \chi(y_i)$, $A_1^e = \chi(T_1) \setminus C^e$, and $A_2^e = \chi(T_2) \setminus C^e$. Since $\tau^*((T, \chi)) = \mathcal{S}$, it holds that $\chi(x) = C$. Let $c \in C$. Since $C$ is a minimal cutset, it follows that there exists $a \in A_i = \chi(T_2) \setminus \chi(x)$ such that $ac$ is an edge of $G$. By the definition of tree decomposition, there exists $t \in V(T)$ such that $\{a, c\} \subseteq \chi(t)$. Since $a \in \chi(T_2) \setminus \chi(x)$, it holds that $t \in V(T_2)$. Therefore, $c \in \chi(T_2)$, and thus $c \in \chi(y_i)$. It follows that $\chi(x) \subseteq \chi(y_i)$, and so $C^e = C$. Now, $S_e = (A_i, \bigcup_{1\leq j \neq i \leq k} A_j, C)$, and so $\{S_{xy_i} \mid 1 \leq i \leq k\} = \rho(S_x)$. This proves \eqref{step-1}.\\ 

For every $x \in X$, let $\S_x = \{S_{xy_i} \mid y_i \in N(x)\}$. Since every edge of $T$ is incident with a vertex in $X$, it follows that $\bigcup_{x \in X} \S_x = \tau((T, \chi))$. By \eqref{step-1}, it follows that $\bigcup_{x \in X} \S_x = \bigcup_{S \in \S} \rho(S) = \rho(\S)$. Therefore, $\tau((T, \chi)) = \rho(S)$. 
\end{proof}

\section*{Acknowledgment}
The author thanks Paul Seymour for asking about how many-sided separations relate to tree decompositions, and Maria Chudnovsky for helpful discussions.

\end{document}